\documentclass[12pt]{amsart}
\usepackage{amssymb}
\newtheorem{theorem}{Theorem}
\newtheorem{lemma}[theorem]{Lemma}
\newtheorem{proposition}[theorem]{Proposition}
\newtheorem{corollary}[theorem]{Corollary}
\newtheorem{example}[theorem]{Example}
\newtheorem{question}{Question}
\newtheorem{definition}[theorem]{Definition}

\title{On the cardinality of almost discretely Lindel\"of spaces}

\author{Angelo Bella}

\address{
Department of Mathematics and Computer Science \\
University of Catania \\
Citt\'a universitaria\\  
viale A. Doria 6 \\
95125 Catania, Italy}

\email{bella@dmi.unict.it}

\author{Santi Spadaro}

\address{Instituto de Matematica e Estatistica (IME-USP) \\ Universidade de Sao Paulo \\ Rua do Matao, 1010 - Cidade Universitaria \\ 05508-090 Sao Paulo - SP \\ Brazil}

\curraddr{Department of Mathematics and Computer Science \\
University of Catania \\
Citt\'a universitaria\\  
viale A. Doria 6 \\
95125 Catania, Italy}

\email{santidspadaro@gmail.com}

\subjclass[2000]{Primary: 54A25, 54D20; Secondary: 54D35, 54D10, 54D55}

\keywords{Cardinal inequality, Lindel\"of space, Arhangel'skii Theorem, elementary submodel, left-separated set, right-separated set, discrete set, free sequence}

\begin{document}
\maketitle

\begin{abstract} 
A space is said to be \emph{almost discretely Lindel\"of} if every discrete subset can be covered by a Lindel\"of subspace. In \cite{JTW}, Juh\'asz, Tkachuk and Wilson asked whether every almost discretely Lindel\"of first-countable Hausdorff space has cardinality at most continuum. We prove that this is the case under $2^{<\mathfrak{c}}=\mathfrak{c}$ (which is a consequence of Martin's Axiom, for example) and for Urysohn spaces in ZFC, thus improving a result by Juh\'asz, Soukup and Szentmikl\'ossy from \cite{JSS}. We conclude with a few related results and questions.
\end{abstract}

\section{Introduction}

In \cite{A} Arhangel'skii published his celebrated theorem stating that every Lindel\"of first-countable Hausdorff space has cardinality at most continuum. Besides solving a long standing question due to Alexandroff and Urysohn, Arhangel'skii's Theorem gave a definite boost to the area of cardinal invariants in topology, inspiring new techniques, results and questions that continue to be the object of current research (see \cite{H} for a survey on Arhangel'skii's Theorem and its legacy). 

Recall that a space is said to be \emph{discretely Lindel\"of} if the closure of every discrete set is Lindel\"of. A well-known question due to Arhangel'skii asks whether every regular discretely Lindel\"of space is Lindel\"of.

A space $X$ is defined to be \emph{almost discretely Lindel\"of} \cite{JTW} if for every discrete set $D \subset X$ there is a Lindel\"of subspace $L$ of $X$ such that $D \subset L$. Of course every discretely Lindel\"of space is almost discretely Lindel\"of. Any example of an $S$-space (a regular hereditarily separable non-Lindel\"of space) provides an (alas, only consistent) example of an almost discretely Lindel\"of non-Lindel\"of regular space. It is still open whether there exists an example of such a space in ZFC.

The authors of \cite{JTW} prove that every almost discretely Lindel\"of first-countable space has cardinality at most $2^{\mathfrak{c}}$ and ask whether every almost discretely Lindel\"of first-countable Hausdorff space has cardinality bounded by the continuum. We prove that this is the case for Urysohn spaces and for Hausdorff spaces under $2^{<\mathfrak{c}}=\mathfrak{c}$. Actually we prove a little more than that, namely: 1) every almost discretely Lindel\"of sequential space such that $\psi(X) \leq \mathfrak{c}$ has cardinality at most continuum and 2) every almost discretely Lindel\"of Hausdorff space such that $\psi_c(X) \cdot t(X)=\omega$ has cardinality at most continuum. 

We should mention that in \cite{JSS} Juh\'asz, Soukup and Szentmikl\'ossy proved that every almost discretely Lindel\"of first-countable \emph{regular} space has cardinality bounded by the continuum.

We conclude by exploring a few further generalizations and related results.

In our proofs we will often use elementary submodels of the structure $(H(\mu), \epsilon)$. Dow's survey \cite{D} is enough to read our paper, and we give a brief informal refresher here. Recall that $H(\mu)$ is the set of all sets whose transitive closure has cardinality smaller than $\mu$. When $\mu$ is regular uncountable, $H(\mu)$ is known to satisfy all axioms of set theory, except the power set axiom. We say, informally, that a formula is satisfied by a set $S$ if it is true when all bounded quantifiers are restricted to $S$. A set $M \subset H(\mu)$ is said to be an elementary submodel of $H(\mu)$ (and we write $M \prec H(\mu)$) if a formula with parameters in $M$ is satisfied by $H(\mu)$ if and only if it is satisfied by $M$. 

The downward L\"owenheim-Skolem theorem guarantees that for every $S \subset H(\mu)$, there is an elementary submodel $M \prec H(\mu)$ such that $|M| \leq |S| \cdot \omega$ and $S \subset M$. This theorem is enough in many applications, but it is often useful (especially in cardinal bounds for topological spaces) to have the following closure property. We say that $M$ is $\kappa$-closed if for every $S \subset M$ such that $|S| \leq \kappa$ we have $S \in M$. For every countable set $S \subset H(\mu)$ there is always a $\kappa$-closed elementary submodel $M \prec H(\mu)$ such that $|M|=2^{\kappa}$ and $S \subset M$.

The following theorem is also used often: let $M \prec H(\mu)$ such that $\kappa + 1 \subset M$ and $S \in M$ be such that $|S| \leq \kappa$. Then $S \subset M$.

All spaces under consideration are assumed to be $T_1$. Undefined notions can be found in \cite{E} for topology and \cite{Ku} for set theory. Our notation regarding cardinal functions mostly follows \cite{J}. In particular, $\psi(X)$ and $t(X)$, denote the pseudocharacter and tightness of $X$ respectively. We recall the definition of these two important cardinal functions, given that they are essential for many of the results in our paper.

Let $A$ be a subset of $X$. The \emph{pseudocharacter of $A$ in $X$} ($\psi(A,X)$) is defined as the minimum cardinal $\kappa$ such that $A$ is the intersection of $\kappa$ many open sets. We denote $\psi(\{x\}, X)$ by $\psi(x,X)$. The pseudocharacter of the space $X$ is defined as $\psi(X)=\sup \{\psi(x,X): x \in X \}$.

The \emph{tightness of the point $x$ in the space $X$} ($t(x,X)$) is defined as the minimum cardinal $\kappa$ such that for every set $A \subset X$ such that $x \in \overline{A} \setminus A$ there is a $\leq \kappa$-sized set $B \subset A$ such that $x \in \overline{B}$. The tightness of the space $X$ is defined as $t(X)=\sup \{t(x,X): x \in X \}$.

\section{The main results}

Recall that a space is right-separated if and only if it admits a well-ordering where every initial segment is open. It is well-known and easy to prove that a space is right-separated if and only if it is \emph{scattered} (that is, every non-empty subset contains an isolated point). We denote by $h(X)$ the supremum of the cardinalities of the \emph{right-separated} subsets of $X$. As is well known (see, for example, \cite{J}, 2.9), $hL(X)=h(X)$ for every space $X$, where $hL(X)$ denotes the hereditarily Lindel\"of degree of $X$.

We denote by $g(X)$ the supremum of the cardinalities of the closures of discrete sets in $X$. Since every scattered space has a dense discrete subset we have $h(X) \leq g(X)$.

We start with an observation contained in the proof of Theorem 4 from \cite{JSS} which we would like to isolate for the convenience of the reader.

\begin{lemma} \label{hlemma} \cite{JSS}
Let $X$ be an almost discretely Lindel\"of $T_2$ space $X$. Then $hL(X) \leq 2^{\chi(X)}$.
\end{lemma}

\begin{proof}
Recall (see for example 2.5 of \cite{J}) that $|Y| \leq d(Y)^{\chi(Y)}$ for every Hausdorff space $Y$ and hence $|\overline{D}| \leq |D|^{\chi(X)}$ for every $D \subset X$. But since $X$ is almost discretely Lindel\"of, for every discrete set $D \subset X$, there is a Lindel\"of space $L \subset X$ such that $D \subset L$. It follows that $|\overline{D}| \leq 2^{\chi(X)}$, by Arhangel'skii's Theorem. Taking suprema we obtain that $g(X) \leq 2^{\chi(X)}$. Now $hL(X) \leq g(X)$.
\end{proof}

\begin{definition} \cite{V}
For any space $X$ and any set $A\subseteq X$, $Cl_\theta(A)$ is the set of all points $x$ such that $\overline{U} \cap A \neq \emptyset $ for every neighbourhood $U$ of $x$. $A$ is $\theta$-closed if $A=Cl_\theta(A)$. The $\theta$-closure $[A]_\theta$ is the smallest $\theta$-closed set containing $A$.
\end{definition}

\begin{lemma}\label{thetaclosurepsi} Let $X$ be an almost discretely
Lindel\"of Hausdorff space, and $A$ be a $\theta$-closed set. Then $\psi(A,X) \leq 2^{\chi(X)}$, that is $A$ is the intersection of a family of $2^{\chi(X)}$ open sets.
\end {lemma}

\begin{proof} 
For each $x \in X \setminus A$ we may fix an open neighbourhood $U_x$ of $x$ such that $\overline {U_x}\cap A=\emptyset$. By Lemma $\ref{hlemma}$, there is a set $S\subseteq X\setminus A$ such that $|S|\leq 2^{\chi(X)}$ and $\bigcup\{U_x : x\in S\}=X\setminus A$. Then $A=\bigcap \{X\setminus \overline {U_x}:x \in S\}$. 
\end{proof} 

\begin{definition}
We say that a sequence $\{x_\alpha :\alpha <\kappa \}$  is \emph{$\theta$-free} if:

$$[\{x_\alpha: \alpha < \beta \}]_\theta \cap
\overline{ \{x_\alpha: \beta \leq \alpha <\kappa\}}=\emptyset$$

\noindent for every $\beta < \kappa$. The cardinal function $F_\theta(X)$ denotes the supremum of the cardinalities of all $\theta$-free sequences contained in $X$.
\end{definition}

Recall that a sequence $\{x_\alpha: \alpha < \kappa \}$ is called \emph{free} if 
$$\overline{\{x_\alpha: \alpha < \beta \}} \cap \overline{\{x_\alpha: \beta \leq \alpha < \kappa \}}=\emptyset$$

\noindent for every $\beta < \kappa$. The cardinal function $F(X)$ is defined as the supremum of the cardinalities of free sequences contained in $X$. Free sequences are an important tool in Arhangel'skii's original solution of the Alexandroff-Urysohn problem.

Obviously, every $\theta$-free sequence is free, so $F_\theta(X) \leq F(X)$ and $F(X)=F_\theta(X)$ for every regular space $X$. However, the inequality may be strict for non-regular spaces (see Section $\ref{OAE}$).

The following lemma is proved via a simple standard argument.

\begin{lemma}\label{freeseq} In an almost discretely Lindel\"of space $X$, every free sequence has length at most $\chi(X)$. 
\end{lemma} 

Recall that the \emph{closed pseudocharacter of the point $x$ in the space $X$} ($\psi_c(x,X)$) is defined as the minimum cardinal $\kappa$ such that there is a family $\{U_\alpha: \alpha < \kappa \}$ of neighbourhoods of $x$ such that $\bigcap \{\overline{U_\alpha}: \alpha < \kappa \}=\{x\}$. The closed pseudocharacter of the space $X$ is then defined as $\psi_c(X)=\sup \{\psi_c(x,X): x \in X \}$. It is easy to see that $\psi_c(X) \leq \chi(X)$, for every Hausdorff space $X$. 

Recall that a space $X$ is \emph{Urysohn} if for every pair of distinct points $x, y \in X$ there are open neighbourhoods $U$ of $x$ and $V$ of $y$ such that $\overline{U} \cap \overline{V}=\emptyset$.

The proof of the following theorem is a variation on the proof of Theorem 2.3 from \cite{B}.

\begin{theorem} \label{mainfree}
Let $X$ be a space such that $F_\theta(X)  \leq
\kappa$ and $\psi([F]_\theta, X) \leq 2^\kappa$, for every
$\theta$-free sequence $F \subset X$. Then there exists $A\subset
X$ such that   $|A| \leq 2^\kappa$ and $X=[A]_\theta$.
\end{theorem}

\begin{proof}
Let $\mu$ be a large enough regular cardinal and $M$ be a
$\kappa$-closed elementary submodel of $H(\mu)$ such that $X \in
M$, $|M|=2^\kappa$ and $2^\kappa+1 \subset M$.

\noindent {\bf Claim 1.} Let $F$ be a $\theta$-free sequence in
$X$ contained in $X \cap M$ and let $p$ be a point outside of
$[F]_\theta$. Then there is an open set $U \in M$ such that
$[F]_\theta \subset U$ and $p \notin U$.

\begin{proof}[Proof of Claim 1]

Since $F$ is a $\theta$-free sequence in $X$, the cardinality of
$F$ does not exceed $\kappa$, so $F \in M$. Hence, by
elementarity, we also have that $[F]_\theta \in M$. Now
$\psi([F]_\theta, X) \leq 2^\kappa$, so we can fix an open family
$\mathcal{U} \in M$ having cardinality $2^\kappa$, such that
$[F]_\theta=\bigcap \mathcal{U}$. Note that $\mathcal{U} \subset
M$. Now pick $U \in \mathcal{U}$ such that $p \notin U$ and note
that $U$ satisfies the requirement of Claim 1.
\renewcommand{\qedsymbol}{$\triangle$}
\end{proof}

\noindent {\bf Claim 2.} $X=[X \cap M]_\theta $. 

\begin{proof}[Proof of Claim 2]
Suppose that the statement of the claim is false and pick a point
$x\notin [X\cap M]_\theta$. Then we can
inductively find points $\{x_\alpha: \alpha < \kappa^+\} \subset
X\cap M$ and open sets $\{U_\alpha: \alpha < \kappa^+\} \subset
M$ such
that:

\begin{enumerate}
\item $[\{x_\alpha: \alpha < \beta\}]_\theta \subset U_\beta$.
\item $x \notin U_\beta$.
\item $x_\beta \notin \bigcup \{U_\alpha: \alpha \leq \beta \}$,
for every $\beta < \kappa^+$.
\end{enumerate}

To see that, suppose for a given $\delta < \kappa^+$ we have
constructed $\{(x_\alpha, U_\alpha): \alpha < \delta \}$
satisfying the three conditions above up to $\delta$. Note that
$\{x_\alpha: \alpha < \delta\}$ is a $\theta$-free sequence.
Indeed, from the first and third condition it follows that
$[\{x_\alpha: \alpha < \gamma \}]_\theta \subset U_\gamma$ and
$\{x_\alpha: \gamma \leq \alpha < \delta \} \subset X \setminus
U_\gamma$, which implies that $[\{x_\alpha: \alpha < \gamma
\}]_\theta \cap \overline{\{x_\alpha: \gamma \leq \alpha < \delta
\}}=\emptyset$.

Therefore, we can use Claim 1 to find an open set $U_\delta$ such
that $[\{x_\alpha: \alpha < \delta \}]_\theta \subset U_\delta$
and $x \notin U_\delta$.

By elementarity,  we
are now allowed to pick a point $x_\delta \in X\cap M  \setminus
\bigcup
\{U_\alpha: \alpha \leq \delta \}$ and thus continue the
induction.

It is easy to see that eventually $\{x_\alpha: \alpha <
\kappa^+\}$ is a $\theta$-free sequence in $X$ of cardinality
$\kappa^+$, which contradicts $F_\theta(X) \leq \kappa$.
\renewcommand{\qedsymbol}{$\triangle$}
\end{proof}
\end{proof}

Since in a Urysohn space $X$ the inequality $|[A]_\theta|\leq
|A|^{\chi(X)}$ holds for any set $A\subset X$ (see \cite{BC}), we immediately
get:
\begin{corollary} If $X$ is a Urysohn almost discretely
Lindel\"of space,then $|X| \le 2^{\chi(X)}$. 
\end{corollary}

Recall that a space is sequential if every non-closed set contains a sequence converging outside of it. It is easy to see that a closure of a subspace in a sequential space is obtained by iterating the sequential closure at most $\omega_1$ many times. Every first-countable space is sequential and every sequential space has countable tightness.

\begin{theorem} \label{mainthm}
($2^{<\mathfrak{c}} = \mathfrak{c}$). Let $X$ be a $T_2$ sequential almost discretely Lindel\"of space such that $\psi(X) \leq \mathfrak{c}$. Then $|X| \leq \mathfrak{c}$.
\end{theorem}

\begin{proof}
Let $\mu$ be a large enough regular cardinal. Let $M$ be a $<\mathfrak{c}$-closed elementary submodel of $H(\mu)$ such that $|M| = \mathfrak{c}$, $\mathfrak{c}+1 \subset M$ and $X \in M$.

From the fact that $X$ is a Hausdorff sequential space and the fact that $M$ is $\omega$-closed it follows that $X \cap M$ is a closed subset of $X$.

We claim that $d(X) \leq \mathfrak{c}$, and that would finish the proof, because every sequential space of density continuum has cardinality continuum.

Suppose by contradiction that $d(X) \geq \mathfrak{c}^+$. Using that, it is easy to find a left-separated subset $L$ of $X$ having cardinality $\mathfrak{c}^+$. Without loss we can assume that $L \in M$.

Since $L$ has cardinality larger than the continuum, we can pick a point $p \in L \setminus M$. Fix a point $x \in X \cap M$. Then we can find a family $\mathcal{U}_x \in M$ of cardinality continuum such that $\bigcap \mathcal{U}_x=\{x\}$. Since $|\mathcal{U}_x| \leq \mathfrak{c}$ and $\mathfrak{c}+1 \subset M$ we actually have that $\mathcal{U}_x \subset M$. Hence, for every $x \in X \cap M$, we can find an open set $U_x \in M$ such that $x \in U_x$ and $p \notin U_x$.

Now $\mathcal{U}=\{U_x: x \in X \cap M \}$ is an open cover of $X \cap M$.

\noindent {\bf Claim.} There is a $<\mathfrak{c}$-sized subcollection of $\mathcal{U}$ covering $L \cap M$.

\begin{proof}[Proof of Claim]
If $L \cap M$ had cardinality smaller than the continuum, this would be trivially true. So we can assume that $|L \cap M| =\mathfrak{c}$.

Let $\{U_\alpha: \alpha < \mathfrak{c}\}$ be an enumeration of $\mathcal{U}$ in type $\mathfrak{c}$ and set $V_\alpha=U_\alpha \setminus \bigcup \{U_\beta: \beta < \alpha\}$. Suppose by contradiction that the statement of the Claim is not true. Then the set $S=\{\alpha < \mathfrak{c}: V_\alpha \cap (L \cap M) \neq \emptyset \}$ has cardinality continuum.

Pick a point $x_\alpha \in V_\alpha \cap L$, for every $\alpha \in S$. Then $R=\{x_\alpha: \alpha \in S \}$ is a set of size continuum which is both right-separated and left-separated. So by 2.12 of \cite{J} the set $R$ contains a discrete set $D$ having cardinality continuum.

Since $X$ is almost discretely Lindel\"of, we can find a Lindel\"of subspace $Y\subset X$. such that $D \subset Y$. Now, $X \cap M$ being closed, the set $Y \cap M$ is also Lindel\"of, and since $\mathcal{U}$ covers $Y \cap M$, we can find an ordinal $\delta < \mathfrak{c}$ such that $D \subset Y \cap M \subset \bigcup \{U_\alpha: \alpha < \delta\}$. But since $D$ has cardinality continuum, there must be $\gamma > \delta$ such that $D \cap V_\gamma \neq \emptyset$ and this contradicts the fact that $V_\gamma$ is disjoint from $\bigcup \{U_\alpha: \alpha < \delta\}$.
\renewcommand{\qedsymbol}{$\triangle$}
\end{proof}

Fix a subcollection $\mathcal{V} \subset \mathcal{U}$ of cardinality smaller than the continuum such that $L \cap M \subset \bigcup \mathcal{V}$.

Since $M$ is $<\mathfrak{c}$-closed we have that $\mathcal{V} \in M$ and since $L$ is also an element of $M$ it follows that: $M \models L \subset \bigcup \mathcal{V}$.

By elementarity $H(\mu) \models L \subset \bigcup \mathcal{V}$, but that is a contradiction because $p \in L \setminus \bigcup \mathcal{V}$.

\end{proof}

\begin{theorem} \label{otherthm} ($2^{<\mathfrak{c}}=\mathfrak{c}$)
Let $X$ be an almost discretely Lindel\"of Hausdorff space such that $\psi_c(X) \cdot t(X)=\omega$. Then $|X| \leq \mathfrak{c}$.
\end{theorem}

\begin{proof}
Let $M$ be given as in the proof of Theorem $\ref{mainthm}$. The proof is essentially the same as the proof of Theorem $\ref{mainthm}$, except that the argument proving that $X \cap M$ is closed is different. Here it is. Let $x \in \overline{X \cap M}$. Pick a family $\{U_n: n < \omega \}$ of neighbourhoods of $x$ such that $\bigcap \{\overline{U_n}: n < \omega \}=\{x\}$. Use the fact that the tightness of $X$ is countable to pick a countable set $C \subset X \cap M$ such that $x \in \overline{C}$. Note that, since $M$ is $\omega$-closed the set $U_n \cap C$ belongs to $M$, for every $n<\omega$. From elementarity and $\omega$-closedness of $M$ again it follows that $\{x\}=\bigcap \{\overline{U_n \cap C}: n <\omega \}$ is also an element of $M$. Hence $x \in X \cap M$ and this concludes the proof that $X \cap M$ is closed.
\end{proof}

\begin{question}
Are Theorems $\ref{mainthm}$ and $\ref{otherthm}$ true in ZFC?
\end{question}

\section{Odds and ends} \label{OAE}

A \emph{cellular family} is a family of pairwise disjoint non-empty open sets. The cellularity of $X$ ($c(X)$) is defined as the supremum of the cardinalities of the cellular families in $X$.

The following is a natural generalization of the notion of an almost discretely Lindel\"of space.

\begin{definition}
We define a space $X$ to be cellular-Lindel\"of if for every cellular family $\mathcal{U}$ there is a Lindel\"of subspace $L \subset X$ such that $U \cap L \neq \emptyset$, for every $U \in \mathcal{U}$.
\end{definition}

The following proposition follows immediately from the definition

\begin{proposition} {\ \\}
\begin{enumerate}

\item Every ccc space is cellular-Lindel\"of.
\item Every Lindel\"of space is cellular-Lindel\"of. 
\item Every almost discretely Lindel\"of space is cellular-Lindel\"of. 
\end{enumerate}
\end{proposition}

So the cellular-Lindel\"of property turns out to be a common weakening of the countable chain condition and the Lindel\"of property, in a similar vein as the weak Lindel\"of property (see \cite{BGW}).

\begin{theorem} \label{cell}
Let $X$ be a Hausdorff cellular-Lindel\"of first-countable space. Then $c(X) \leq \mathfrak{c}$.
\end{theorem}

\begin{proof}
Let $\mathcal{U}$ be a cellular family. Since $X$ is cellular-Lindel\"of we can find a Lindel\"of space $L \subset X$ such that $L \cap U \neq \emptyset$, for every $U \in \mathcal{U}$. But every Lindel\"of first-countable space has cardinality at most the continuum, so $|L| \leq \mathfrak{c}$ and hence $|\mathcal{U}| \leq \mathfrak{c}$.
\end{proof}

\begin{corollary}
Every Hausdorff cellular-Lindel\"of first countable space has cardinality at most $2^{\mathfrak{c}}$.
\end{corollary}

\begin{proof}
This is an immediate consequence of the Hajnal-Juh\'asz inequality $|X| \leq 2^{\chi(X) \cdot c(X)}$ (see, for example, \cite{J}, 2.15 b)) and Theorem $\ref{cell}$.
\end{proof}

\begin{example}
There are cellular-Lindel\"of non-linearly Lindel\"of Tychonoff spaces in ZFC.
\end{example}

\begin{proof}
Let $X=\Sigma(2^\kappa)=\{x \in 2^{\kappa}: |x^{-1}(1)| \leq \aleph_0 \}$ with the topology induced from the usual product topology on $2^\kappa$. Then $X$ is ccc and hence it's cellular Lindel\"of. Moreover $X$ is countably compact, so it can't be linearly Lindel\"of, or otherwise it would be compact.
\end{proof}

The above example should be contrasted with the fact that no example of a regular almost discretely Lindel\"of non-Lindel\"of space is known in ZFC. Also it seems that not even a consistent $T_1$ example of a discretely Lindel\"of non-Lindel\"of space is known at the moment.

Recall that a space is weakly Lindel\"of if every open cover has a countable subcollection with a dense union (see \cite{BGW}).

\begin{example} \label{wLncL}
There is a weakly Lindel\"of $T_2$ non-cellular Lindel\"of space.
\end{example}

\begin{proof}
This is Example 2.3 from \cite{BGW}. Let $\kappa$ be a cardinal larger than the continuum and let $A$ be a countable dense subset of the irrationals. Define a topology on $X=(\mathbb{Q} \times \kappa) \cup A$ by declaring a basic neighbourhood of a point $(x, \alpha)$, where $x \in \mathbb{Q}$ and $\alpha < \kappa$,  to be $(U \cap \mathbb{Q}) \times \{\alpha\}$, where $U$ is an open Euclidean interval containing $x$ and a basic neighbourhood of a point $y \in A$ to be of the form $(U \cap A) \cup ((U \cap \mathbb{Q}) \times \kappa)$. This space is weakly Lindel\"of because every open set containing $A$ is dense in $X$ (see \cite{BGW}). However, it is not cellular-Lindel\"of because it is first-countable but $c(X)>\mathfrak{c}$.
\end{proof}

Note that Example $\ref{wLncL}$ is only Hausdorff.

\begin{question}
Is there a Tychonoff example of a weakly Lindel\"of non-cellular Lindel\"of space?
\end{question}

\begin{question}
Is there a cellular-Lindel\"of non-weakly Lindel\"of space?
\end{question}

\begin{question}
Is it true that every first-countable cellular-Lindel\"of Hausdorff space has cardinality at most continuum?
\end{question}

We conclude with a few more applications of the notion of $\theta$-free sequence. But first of all, let us note how the inequality $F_\theta(X) < F(X)$ can occur even for \emph{Urysohn spaces} (that is spaces where each pair of distinct points can be separated by open neighbourhoods with disjoint closures).

\begin{example} A Urysohn space $X$ such that
$F_\theta(X)<F(X)$. \end{example}

\begin{proof}  
Let $X=K (\omega)$ be the Kat\v{e}tov extension of $\omega$. Recall that  the underlying set of $K(\omega)$ is the same as the \v{C}ech-Stone compactification of the integers $\beta \omega$, that is, the set of all ultrafilters on $\omega$ (principal ultrafilters are identified with points of $\omega$ in the obvious way) but a local base at $p \in K(\omega)$ in the Kat\v{e}tov extension is given by $\{\{p\} \cup A : A\in p\}$.  Note that $K(\omega)\setminus \omega$ is a closed discrete set of cardinality $2^\mathfrak{c}$. Therefore, $F(K(\omega))=2^\mathfrak{c}$. 

The topology of $K(\omega)$ is finer than the topology of $\beta \omega$, however, for every $p \in K(\omega) \setminus \omega$, the topologies induced on $\{p\} \cup \omega$ by $K(\omega)$ and $\beta \omega$ are the same. Combining this with the observation that $\omega$ is dense in $K(\omega)$ we see that, for every open set $U \subset K(\omega)$, we have $\overline
U=Cl_{K(\omega)}(U)=Cl_{K(\omega)}(U\cap
\omega)=Cl_{\beta\omega}(U\cap \omega)$. Therefore the closure of an open set in $K(\omega)$ is actually a clopen set in $\beta \omega$.  

 Let
$S=\{x_\alpha :\alpha <\kappa \}$ be a $\theta$-free  
sequence in
$K(\omega)$.   Fix $\alpha
<\kappa $. For every $\gamma\in \kappa \setminus \alpha $ there
exists an open neighbourhood $U_\gamma$ of $x_\gamma$ such that
$\overline {U_\gamma}\cap \{x_\beta:\beta<\alpha \}=\emptyset $.
Since the set $\bigcup\{\overline{U_\gamma}:\gamma\in \kappa \setminus
\alpha \}$ is open in $\beta \omega$, we see that  $S$ is a left-separated set in $\beta \omega$.  Thus, we have that $|S| \le \mathfrak c$ and hence $F_\theta(K(\omega))\le \mathfrak c$.
\end{proof}

Given a space $X$, a set $A\subseteq X$ is $\theta$-dense in $X$
if $[A]_\theta=X$. The $\theta$-density $\theta d(X)$ is the smallest cardinality of
a $\theta$-dense subset of $X$.

 \begin{theorem} \label{thmcaliber}  Let $X$  be a space. If 
a cardinal $\lambda $ satisfying 
$F_\theta(X)<\lambda \le \left(2^{F_\theta(X)}\right)^+$
is a caliber of $X$, then $\theta d(X)\le 2^{F_\theta(X)}$.
\end{theorem}

\begin{proof} Let $F_\theta(X)=\kappa $ and assume by
contradiction that
the $\theta$-density of $X$ is bigger than $2^\kappa $.  Fix a
choice
function $\eta :\mathcal P(X)\setminus \{\emptyset \}\rightarrow
X$.
We
will define by induction  an increasing family $\{A_\alpha
:\alpha
<\lambda  \}$ of subsets of $X$ of cardinality not exceeding
$2^\kappa $ and a family $\{U_\alpha :\alpha <\lambda  \}$ of
non-empty  open subsets of $X$ in such a way that:

\begin{enumerate}

\item $[A_\alpha ]_\theta \cap \overline {U_\alpha
}=\emptyset$;
\item if $\mathcal V\subseteq \{U_\beta:\beta<\alpha \}$
satisfies
$|\mathcal V|\le \kappa $ and $\bigcap \mathcal V\ne \emptyset$,
then
$\eta (\bigcap\mathcal V)\in A_\alpha
$.

\end{enumerate}

To justify the inductive construction, let us assume to have
already defined the sets $\{A_\beta :\beta<\alpha \}$ and
$\{U_\beta :\beta<\alpha \}$.  Since $\alpha <\lambda  $ and
$ \lambda \le \left(2^{F_\theta(X)}\right)^+$, we have 
$|\{U_\beta
:\beta<\alpha \}|\le|\alpha| \le 
2^{\kappa }$. Consequently,  the set $B=\{\eta (\bigcap \mathcal
V):
\mathcal
V\subseteq 
\{U_\beta :\beta<\alpha \}, |\mathcal V|\le \kappa $ and $\bigcap
\mathcal V\ne \emptyset \}$ has cardinality not
exceeding $2^\kappa $. Then, let $A_\alpha
=B\cup\bigcup\{A_\beta :\beta<\alpha \}$. 
As we are assuming that the $\theta$-density of $X$ is bigger
than
$2^\kappa  $,  we may find a non-empty open set $U_\alpha $
such that $ [A_\alpha ]_\theta\cap \overline {U_\alpha
}=\emptyset $. 

Since
$\lambda $ is a caliber of $X$, there exists  a set $S\subseteq
\lambda $ such that $|S|=\lambda $ and $\bigcap\{U_\alpha :\alpha
\in S\}\ne \emptyset $. We may fix     an increasing
mapping $f:\lambda  \rightarrow S$. Observe now that we are
assuming $\kappa ^+\le \lambda $.    For any $\alpha
<\kappa ^+$ let $x_\alpha =\eta (\bigcap \{U_{f(\xi)} :\xi\le
\alpha
\})$.  We claim that  the set $\{x_\alpha  :\alpha <\kappa ^+\}$
so obtained 
is a $\theta$-free sequence in $X$. To check this,  fix $\alpha 
<\kappa
^+$ and observe that for each $\beta <\alpha $ we have $x_\beta
\in A_{f(\beta)+1}\subseteq  A_{f(\alpha )}$.  Moreover, for each
$\beta\ge \alpha $ the
set $U_{f(\alpha )}$ occurs in the definition of $x_\beta$ and
consequently $x_\beta \in U_{f(\alpha )}$.  This means that
$\{x_\beta :\beta<\alpha \}\subseteq A_{f(\alpha )}$ and
$\{x_\beta:\alpha \le \beta<\kappa ^+\}\subseteq U_{f(\alpha )}$.
Therefore $ [\{x_\beta :\beta<\alpha \}]_\theta\cap
\overline {\{x_\beta :\alpha \le \beta<\kappa ^+\}}\subseteq
[A_{f(\alpha )}]_\theta \cap \overline {U_{f(\alpha
)}}=\emptyset$.
The validity of the claim contradicts the hypothesis and  the
proof is then complete. 

\end{proof}

\begin{corollary}  
Let $X$ be a regular space. If a cardinal $\lambda$ satisfying $F(X) < \lambda \leq (2^{F(X)})^+$
is a caliber of $X$, then $d(X) \leq 2^{F(X)}$
\end{corollary}

\begin{corollary}
Let $X$ be a regular sequential space with no uncountable free sequence and $\lambda \leq \mathfrak{c}^+$ be an uncountable cardinal such that $\lambda$ is a caliber of $X$. Then $|X| \leq \mathfrak{c}$.
\end{corollary}

\section{Acknowledgements}

The first-named author was partially supported by a grant of the Group GNSAGA of INDAM. The second-named author is grateful to FAPESP for financial support through postdoctoral grant 2013/14640-1, \emph{Discrete sets and cardinal invariants in set-theoretic topology}. Part of the research for the paper was carried out when he visited the first-named author at the University of Catania in December 2016. He thanks his colleagues there for the warm hospitality. The authors are grateful to Lajos Soukup for spotting an error in an earlier version of the paper.

\end{document}